\documentclass[11pt, reqno]{amsart}
\usepackage{amsmath,amssymb,amsfonts,amsthm,graphicx,mathrsfs,mathtools}
\usepackage[margin=1.5in]{geometry}
\usepackage[dvipsnames]{xcolor}
\makeatletter
\@namedef{subjclassname@2020}{%
\textup{2020} Mathematics Subject Classification}
\makeatother

\title[Lagrangian self-shrinkers with reflection symmetry]{Closed lagrangian self-shrinkers in $\R^4$ symmetric with respect to a hyperplane}

\author[J. Lee]{Jaehoon Lee}

\address[]{Jaehoon Lee, Department of Mathematical Sciences, Seoul National University, Seoul  08826, Korea}

\email{jaehoon.lee@snu.ac.kr}

\begin{document}

\newtheorem{theorem}{theorem}[section]
\newtheorem{thm}[theorem]{Theorem}
\newtheorem{lem}[theorem]{Lemma}
\newtheorem{cor}[theorem]{Corollary}
\newtheorem{prop}[theorem]{Proposition}
\newtheorem{rmk}[theorem]{Remark}
\newtheorem{Def}[theorem]{Definition}
\newtheorem{Question}[theorem]{Question}

\renewcommand{\theequation}{\thesection.\arabic{equation}}
\newcommand{\RNum}[1]{\uppercase\expandafter{\romannumeral #1\relax}}
\newcommand{\R}{\mathbb{R}}
\newcommand{\C}{\mathbb{C}}
\newcommand{\grad}{\nabla}
\newcommand{\laplacian}{\Delta}
\renewcommand{\d}{\textup{d}}


\begin{abstract}
In this paper, we prove that the closed Lagrangian self-shrinkers in $\R^4$ which are symmetric with respect to a hyperplane are given by the products of Abresch-Langer curves. As a corollary, we obtain a new geometric characterization of the Clifford torus as the unique embedded closed Lagrangian self-shrinker symmetric with respect to a hyperplane in $\R^4$. 
\end{abstract}

\maketitle

\section{Introduction}
\setcounter{equation}{0}
A \emph{self-shrinker} of the mean curvature flow is defined to be an immersed submanifold in the Euclidean space, $F : M^n \to \R^{n+k}$, which satisfies the quasilinear elliptic system
	\begin{align}\label{SS}
	\vec{H}=-F^{\perp}.
	\end{align}
Here, $\vec{H}$ is the mean curvature vector given by the trace of the second fundamental form and $\perp$ means the projection onto the normal bundle of $M^n$.

It is well known that the blow-up limit at the Type \RNum{1} singularity of the mean curvature flow is a self-shrinker. Self-shrinkers themselves also provide self-similar solutions of the mean curvature flow that shrink to the origin. Furthermore, if we consider the Euclidean space as a weighted Riemannian manifold with the Gaussian density $e^{-\frac{|x|^2}{2}}$, the solution of (\ref{SS}) corresponds to a minimal submanifold in a weighted sense. For these reasons, it is important to understand the geometry of self-shrinkers.

Abresch and Langer \cite{AL} completely determined closed self-shrinking planar curves. These curves are called Abresch-Langer curves and note that they have reflection symmetry. Moreover, there is a positive constant associated with each curve so that if two Abresch-Langer curves have the same constants, then they are the same up to a rigid motion (see Lemma \ref{TR}). However, unlike the case of curves, the situation becomes more complicated, so it is hard to expect a complete classification in higher dimensions.

On the other hand, the mean curvature flow preserves the Lagrangian condition in K\"{a}hler-Einstein manifolds, including the Euclidean space (see \cite{SmoLag}). Since the mean curvature flow corresponds to the negative gradient flow for a volume functional, it provides a potential way to obtain a volume minimizer. For instance, special Lagrangian submanifolds are volume minimizing in a Calabi-Yau manifold, so the Lagrangian mean curvature flow plays a significant role in the construction of these examples. However, the development of finite-time singularities is the main difficulty of this method (see for instance \cite{Neves}). Therefore, in order to manage possible singularities, it is necessary to study Lagrangian self-shrinkers.

In this paper, we consider Lagrangian self-shrinkers in $\R^4 \simeq \C^2$. Since Lagrangian self-shrinking sphere cannot exist by the theorem of Smoczyk \cite{SmoSph}, the simplest example is the surface of genus 1. Many immersed Lagrangian self-shrinking tori were constructed: products of Abresch-Langer curves, Anciaux's tori (see \cite{Anciaux}), and Lee-Wang's tori (see \cite{LWG}, \cite{LWF}).

One interesting observation we would like to emphasize is that all known embedded examples of closed Lagrangian self-shrinkers in $\C^2$ become the Clifford torus. This is not true in $\C^n$($n \geq 3$) since two or more different embedded Lagrangian self-shrinkers can be found in Anciaux's examples (see \cite{Anciaux}). Moreover, as the normal and tangent bundles of a Lagrangian submanifold are diffeomorphic and the self-intersection number is given by the Euler characteristic of the normal bundle, one can conclude that embedded Lagrangian surfaces in $\C^2$ should have genus 1. Therefore, it is natural to ask whether the Clifford torus is the only embedded example in $\C^2$:
	\begin{Question}\label{conj}
	Is the Clifford torus unique as an embedded Lagrangian self-shrinker in $\C^2$?
	\end{Question}

This question is analogous to Lawson's conjecture that whether the Clifford torus is the only embedded minimal torus in $S^3$. The Lawson conjecture was proved by Brendle \cite{Brendle}, by clever use of the maximum principle for two-point functions on a given surface. However, due to the increased codimension, it is not easy to apply the maximum principle in our case. But, assuming a particular symmetry condition, we obtain a positive answer for Question \ref{conj}.

Before we state the main results, we recall some rigidity results on the Clifford torus as a Lagrangian self-shrinker. Castro and Lerma \cite{CLHminimal} characterized Lee-Wang's tori as compact Hamiltonian stationary Lagrangian self-shrinkers in $\C^2$ and the Clifford torus was characterized as the only embedded example among them. In \cite{CLLag}, they proved that if a compact Lagrangian self-shrinker satisfies either $|\vec{H}|^2=const$ or $|\vec{H}|^2\leq 2$ or $|\vec{H}|^2\geq 2$, then it is the Clifford torus. They also proved that a compact Lagrangian self-shrinker without change of sign on the Gauss curvature and satisfying $|\sigma|^2\leq2$ is the Clifford torus. Here, $\sigma$ denotes the second fundamental form. Then Li and Wang \cite{LiWang} generalized previous results into two directions: the Clifford torus is the unique compact Lagrangian self-shrinker with $|\sigma|^2\leq 2$, and if a compact Lagrangian self-shrinker has no sign change on the Gauss curvature, then it is one of the products of Abresch-Langer curves. 

Now we recall the result of Ros related to Lawson's conjecture, which states that the Clifford torus is the unique embedded minimal torus in $S^3$ symmetric with respect to four pairwise orthogonal hyperplanes in $\R^4$ (see Theorem 6 in \cite{Ros}). Since the almost complex structure does not commute with reflections in general, we could expect that there would be a non-trivial restriction on a Lagrangian submanifold if we assume the reflection symmetry. Motivated by this, we study Lagrangian self-shrinkers symmetric with respect to a hyperplane in $\R^4$. 

Then we have the following theorem.

\begin{thm}\label{MTHM}
Let $F : \Sigma^2 \to \R^4\simeq \C^2$ be a closed Lagrangian self-shrinker symmetric with respect to a hyperplane $P$. Then $\Sigma^2$ is given by the product of two Abresch-Langer curves.
\end{thm}

We remark that it requires only one hyperplane of symmetry in Theorem \ref{MTHM}, while Ros' theorem assumed four orthogonal hyperplanes of symmetry. For the proof of Theorem \ref{MTHM}, we were not able to apply the method of Ros \cite{Ros} due to the high codimension. Instead, we observe that each Lagrangian self-shrinker with reflection symmetry satisfies global relations which are similar to transcendental relations (\ref{treq}) on Abresch-Langer curves (see Proposition \ref{GC}). By using such relations, we prove that $\Sigma^2$ is flat. Then, by the result of Li and Wang (see Proposition \ref{Flat}), we could finish the proof. We also show that global relations on a product of Abresch-Langer curves coincide with the transcendental relation on each curve.

Since the Clifford torus is the only embedded self-shrinker among the products of Abresch-Langer curves, we obtain a new geometric characterization of the Clifford torus as a Lagrangian self-shrinker:

\begin{thm}\label{COR}
A closed embedded Lagrangian self-shrinker symmetric with respect to a hyperplane in $\R^4$ is the Clifford torus.
\end{thm}

The paper is organized as follows. In Section 2 we provide some basic definitions and show that a hyperplane of symmetry can be assumed to be the coordinate hyperplane $\{x_1=0\}$ in a way that preserves the Lagrangian condition. We also recall two previous results that play an important role in this paper. In the next section, we compute local equations for Lagrangian self-shrinkers with reflection symmetry in terms of isothermal coordinates. Then we obtain global relations in Section 4. Finally, we prove the main results in Section 5.

\section*{Acknowledgements}
The author would like to express his gratitude to Jaigyoung Choe for helpful discussions and thoughtful encouragement. This work was supported in part by NRF-2018R1A2B6004262.

\section{Preliminaries}
\setcounter{equation}{0}
Identify $\R^4$ with $\C^2$ via $\left(x_1, x_2, x_3, x_4\right) \leftrightarrow \left(x_1+ix_2, x_3+ix_4\right)$, where $i=\sqrt{-1}$. We use the almost complex structure $J$ on $\R^4$ which corresponds to the multiplication of $i$ in $\C^2$. 
Let $\left\langle\ ,\ \right\rangle$ be the standard hermitian product in $\C^2$. Then the Euclidean inner product in $\R^4$ is given by the real part of $\left\langle\ ,\ \right\rangle$ under the identification.

The K\"{a}hler form $\omega$ on $\left( \R^4, J\right)$ is given by
	\begin{align*}
	\omega=\d x_1\wedge \d x_2+\d x_3\wedge \d x_4,
	\end{align*}
where $x_i$'s are coordinate functions of $\R^4$. An immersion $F : \Sigma^2 \to \R^4$ is called a \emph{Lagrangian} if 
	\begin{align}\label{Lag}
	F^{*}\omega=0,
	\end{align}
and we call $F : \Sigma^2 \to \R^4$ a \emph{Lagrangian self-shrinker} if it satisfies (\ref{SS}) and the Lagrangian condition (\ref{Lag}).

For arbitrary $p\in\Sigma$, one can always find local \emph{isothermal coordinates} $x$, $y$ in a neighborhood of $p$ such that
	\begin{align}\label{Iso}
	|F_x|=|F_y|,\ F_x\perp F_y,
	\end{align}
where the lower indices denote partial derivatives. In terms of the isothermal coordinates, the Lagrangian condition (\ref{Lag}) is equivalent to
	\begin{align}\label{isLag}
	JF_x \perp F_y.
	\end{align}
	
Now let $F : \Sigma^2 \to \R^4$ be a Lagrangian self-shrinker symmetric with respect to a hyperplane $P$. First, we claim that $P$ contains the origin $O$. Indeed, by the reflection symmetry, $\Sigma$ and its reflected surface shrink to the same point along the mean curvature flow. As the surface shrinks to the origin,  $O$ should be invariant under the reflection, which implies that $O \in P$.

Let $\nu\in \R^4\simeq\C^2$ be the unit normal vector of $P$. Then there exists a unitary matrix $G\in U(2)$ such that
	\begin{align*}
	G\cdot \nu=(1, 0) \in \C^2.
	\end{align*}
The unitary group $U(2)$ can be identified as a subgroup of $O(4)$ consisting of elements which commute with $J$, so there is an orthogonal transformation $\tilde{G} \in O(4)$ corresponding to $G$. Then $\tilde{G}$ sends $P$ to the coordinate hyperplane $\{x_1=0\}$. Moreover, it preserves Lagrangian submanifolds since it commutes with $J$. Thus, $\tilde{G}\circ F : \Sigma^2 \to \R^4$ is a Lagrangian self-shrinker symmetric to $\{x_1=0\}$. Therefore it suffices to consider Lagrangian self-shrinkers $F : \Sigma^2 \to \R^4$ symmetric with respect to $\{x_1=0\}$. This family of self-shrinkers will be denoted by $\mathscr{F}$.

Lagrangian self-shrinkers with reflection symmetry can be easily found in the products of Abresch-Langer curves. For later use, we recall two results on Abresch-Langer curves and their product self-shrinkers. The first result is about the \emph{transcendental relation} on those curves (see Theorem A in \cite{AL} and Lemma 5.3 in \cite{SmoAC}).

\begin{lem}[Lemma 5.3 in \cite{SmoAC}]\label{TR}
Let $\Gamma \subset \R^2$ be one of the Abresch-Langer curves. Then there exists a positive constant $c_{\Gamma}$ such that 
	\begin{align}\label{treq}
	ke^{-\frac{r^2}{2}}=c_{\Gamma}
	\end{align}
holds everywhere on $\Gamma$, where $r$ is the distance from the origin and $k$ is the curvature. If two Abresch-Langer curves $\Gamma_1$ and $\Gamma_2$ have the same constants, i.e., $c_{\Gamma_1}=c_{\Gamma_2}$, then (up to a rigid motion)  $\Gamma_1=\Gamma_2$. Moreover, the critical values $k_c$ of the curvature $k$ satisfy $k_ce^{-\frac{k_c^2}{2}}=c_{\Gamma}$.
\end{lem}

The second one is about the characterization of products of Abresch-Langer curves as Lagrangian self-shrinkers with constant Gauss curvature in $\R^4$ (see Proposition 5.6 in \cite{LiWang}).

\begin{prop}[Proposition 5.6 in \cite{LiWang}]\label{Flat}
Let $F : \Sigma^2 \to \R^4$ be a compact Lagrangian self-shrinker with constant Gauss curvature $K$. Then $K=0$ and $\Sigma^2$ is given by the product of two Abresch-Langer curves. 
\end{prop}

Throughout this paper, the gradient and Laplacian on a given surface will be denoted by $\grad$ and $\laplacian$, respectively.

\section{Local Equations}
\setcounter{equation}{0}

Let $F : \Sigma^2 \to \R^4\simeq \C^2$ be a closed Lagrangian self-shrinker symmetric with respect to $\{x_1=0\}$, i.e., $F\in \mathscr{F}$. Using isothermal coordinates, we obtain the following local equations:

\begin{prop}\label{PLE}
Suppose that $F : \Sigma^2 \to \C^2$ is given by $F=(A, B)$ for some complex-valued functions $A$ and $B$. Then
	\begin{enumerate}
		\item \label{(1)}$|A_x|=|B_y|$, $|A_y|=|B_x|$,
		\item \label{(2)}$A_x\bar{A_y}+B_x\bar{B_y}=0$, and $A_x\bar{A_y}$, $B_x\bar{B_y}$ 	are real-valued,
		\item \label{(3)}$(\laplacian A+A)\bar{A_x}$, $(\laplacian A+A)\bar{A_y}$, $(\laplacian B+B)\bar{B_x}$, $(\laplacian B+B)\bar{B_y}$ are real-valued,
	\end{enumerate}
hold true on all of $\Sigma^2$, where $x, y$ are local isothermal coordinates.	
\end{prop}

\begin{proof}
Since the coordinates are assumed to be isothermal, from (\ref{Iso}) we have
	\begin{align}
	F_x \perp F_y &\Leftrightarrow \mbox{Re}(A_x\bar{A_y}+B_x\bar{B_y})=0,\label{conf1}\\
	|F_x|^2=|F_y|^2 &\Leftrightarrow |A_x|^2+|B_x|^2=|A_y|^2+|B_y|^2,\label{conf2}
	\end{align}
and the Lagrangian condition (\ref{isLag}) implies
	\begin{align}
	JF_x \perp F_y &\Leftrightarrow \mbox{Re}(iA_x\bar{A_y}+iB_x\bar{B_y})=0 	\Leftrightarrow \mbox{Im}(A_x\bar{A_y}+B_x\bar{B_y})=0\label{lag}.
	\end{align}
Combining (\ref{conf1}) and (\ref{lag}), we obtain
	\begin{align}
	A_x\bar{A_y}+B_x\bar{B_y}=0\label{eq1},
	\end{align}
and then with (\ref{conf2}), we conclude that
	\begin{align}
	|A_x|=|B_y|,\ |A_y|=|B_x|\label{eq2}.
	\end{align}

On the other hand, since $\Sigma^2$ has the reflection symmetry with respect to $x_1$-hyperplane, the same equations hold when we replace $A$ by $-\bar{A}$. 

\noindent Now substituting $A$ by $-\bar{A}$ in (\ref{eq1}) gives
	\begin{align}
	\Bar{A_x}A_y+B_x\bar{B_y}=0\label{eq11}.
	\end{align}
Hence we get $A_x\bar{A_y}=\bar{A_x}A_y$ and $B_x\bar{B_y}=\bar{B_x}B_y$ from (\ref{eq1}) and (\ref{eq11}), which imply that $A_x\bar{A_y}$ and $B_x\bar{B_y}$ are real-valued.

In order to prove (\ref{(3)}), we compute each term in (\ref{SS}), $\vec{H}=-F^{\perp}$, directly. Let $\lambda\coloneqq|F_x|^2=|F_y|^2$. As $\left\{\frac{1}{\sqrt{\lambda}}F_x, \frac{1}{\sqrt{\lambda}}F_y\right\}$ form an orthonormal frame, the mean curvature vector can be expressed as
	\begin{align}
	\vec{H}&=\left(\laplacian F\right)^{\perp} \nonumber \\
			&=\laplacian F-\frac{\mbox{Re}\langle\laplacian F, F_x\rangle}{\lambda}F_x-\frac{\mbox{Re}\langle\laplacian F, F_y\rangle}{\lambda}F_y \nonumber \\
			&=\left(\laplacian A-\frac{\mbox{Re}(\laplacian A \bar{A_x}+\laplacian B\bar{B_x})}{\lambda}A_x-\frac{\mbox{Re}(\laplacian A\bar{A_y}+\laplacian B\bar{B_y})}{\lambda}A_y,\right. \nonumber \\
			&\left. \ \ \ \ \ \ \ \laplacian B-\frac{\mbox{Re}(\laplacian A\bar{A_x}+\laplacian B\bar{B_x})}{\lambda}B_x-\frac{\mbox{Re}(\laplacian A\bar{A_y}+\laplacian B\bar{B_y})}{\lambda}B_y\right),\label{meancurv}
	\end{align}
and similarly the normal part of the position vector is given by
	\begin{align}
	F^{\perp}&=F-\frac{\mbox{Re}\langle F, F_x\rangle}{\lambda}F_x-\frac{\mbox{Re}\langle F, F_y\rangle}{\lambda}F_y \nonumber \\
			&=\left(A-\frac{\mbox{Re}(A\bar{A_x}+B\bar{B_x})}{\lambda}A_x-\frac{\mbox{Re}(A\bar{A_y}+B\bar{B_y})}{\lambda}A_y,\right. \nonumber \\
			&\left. \ \ \ \ \ \ \ B-\frac{\mbox{Re}(A\bar{A_x}+B\bar{B_x})}{\lambda}B_x-\frac{\mbox{Re}(A\bar{A_y}+B\bar{B_y})}{\lambda}B_y\right).\label{normalpart}
	\end{align}

Let $\tilde{u}\coloneqq\bar{A_x}A_y$, which is a real-valued function. Clearly, $B_x\bar{B_y}=-\tilde{u}$, by (\ref{eq1}). 
From (\ref{eq2}), $|A_y|=|B_x|$, we may write
	\begin{align*}
	\lambda=|A_x|^2+|A_y|^2.
	\end{align*}
Since $F$ is an immersion, $\lambda\neq0$ and either $A_x\neq0$ or $A_y\neq0$ hold. In both cases, a similar argument applies, so we may assume that $A_x\neq0$.
Then we have the following relations:
	\begin{align*}
	A_y=uA_x, \ B_x=-uB_y,
	\end{align*}
where $u\coloneqq \frac{\tilde{u}}{|A_x|^2}$ is also real-valued. By the above relations,
	\begin{align*}
	&\mbox{Re}(\laplacian A\bar{A_x}+\laplacian B\bar{B_x})A_x+\mbox{Re}(\laplacian A\bar{A_y}+\laplacian B\bar{B_y})A_y \\
	&=\mbox{Re}(\laplacian A\bar{A_x}-u\laplacian B\bar{B_y})A_x+\mbox{Re}(u\laplacian A\bar{A_x}+\laplacian B\bar{B_y})(uA_x) \\
	&=(1+u^2)\mbox{Re}(\laplacian A\bar{A_x})A_x \\
	&=\frac{\lambda}{|A_x|^2}\mbox{Re}(\laplacian A\bar{A_x})A_x, 
	\end{align*}
and we compute
	\begin{align*}
	\laplacian A-\frac{\mbox{Re}(\laplacian A\bar{A_x}+\laplacian B\bar{B_x})}{\lambda}A_x-\frac{\mbox{Re}(\laplacian A\bar{A_y}+\laplacian B\bar{B_y})}{\lambda}A_y=\frac{(\laplacian A\bar{A_x}-\laplacian \bar{A}A_x)}{2|A_x|^2}A_x.
	\end{align*}
By applying similar computations to (\ref{meancurv}) and (\ref{normalpart}), we obtain
	\begin{align*}
	\vec{H}=\left(\frac{(\laplacian A\bar{A_x}-\laplacian \bar{A}A_x)}{2|A_x|^2}A_x, \ \frac{(\laplacian B\bar{B_y}-\laplacian \bar{B}B_y)}{2|B_y|^2}B_y\right)
	\end{align*}
and
	\begin{align*}
	F^{\perp}=\left(\frac{(A\bar{A_x}-\bar{A}A_x)}{2|A_x|^2}A_x, \ \frac{(B\bar{B_y}-\bar{B}B_y)}{2|B_y|^2}B_y\right).
	\end{align*}
Consequently, (\ref{SS}), $\vec{H}=-F^{\perp}$, is equivalent to
	\begin{align*}
	(\laplacian A+A)\bar{A_x}=(\laplacian \bar{A}+\bar{A})A_x
	\end{align*}
and
	\begin{align*}
	(\laplacian B+B)\bar{B_y}=(\laplacian \bar{B}+\bar{B})B_y,
	\end{align*}
which imply that $(\laplacian A+A)\bar{A_x}$ and $(\laplacian B+B)\bar{B_y}$ are real-valued. Since $A_x\bar{A_y}$ and $B_x\bar{B_y}$ are also real-valued, (\ref{(3)}) is proved.
\end{proof}

\section{Global Relations}
\setcounter{equation}{0}

In this section, we prove that Lagrangian self-shrinkers in $\mathscr{F}$ have a special property analogous to the transcendental relation on Abresch-Langer curves as in Lemma \ref{TR}. From local equations in Proposition \ref{PLE}, we obtain the following lemma:

\begin{lem}\label{LC}
With the same notation in Proposition \ref{PLE}, suppose that $|A|>0$ in a neighborhood of $p\in \Sigma^2$. Let $A=re^{i\theta}$ be a locally given polar representation of $A$. Then $r^4|\grad \theta|^2=Ce^{r^2}$ for some constant $C$.
\end{lem}

\begin{proof}
By a straightforward calculation, we have
	\begin{align*}
	A_x=\left(r_x+ir\theta_x\right)e^{i\theta},\ A_y=\left(r_y+ir\theta_y\right)e^{i\theta},
	\end{align*}
where $x, y$ are isothermal coordinates.
It then follows from  (\ref{(2)}) in Proposition \ref{PLE}
	\begin{align*}
	\frac{\mbox{Im}(A_x\bar{A_y})}{r}=\theta_xr_y-r_x\theta_y=0, 
	\end{align*}
which implies that $\grad r$ and $\grad \theta$ are linearly dependent.
Moreover, $|\grad r|^2+r^2|\grad \theta|^2=1$ implies that either $|\grad r|\neq 0$ or $|\grad \theta|\neq 0$. We treat both cases separately as follows.

First, assume that $|\grad r|\neq 0$. We may write $\grad \theta=\mu \grad r$ for some function $\mu$. From this we derive
	\begin{align*}
	\laplacian \theta=\mu \laplacian r+\grad \mu \cdot \grad r,
	\end{align*}
and we compute
	\begin{align*}
	\laplacian A+A&=\left(\laplacian r+r-r|\grad \theta|^2+2i\grad r \cdot \grad \theta+ir\laplacian \theta\right)e^{i\theta} \nonumber \\
		&=\left(\laplacian r+r-r\mu^2|\grad r|^2+2i\mu |\grad r|^2+ir\mu \laplacian r+ir\grad r \cdot \grad \mu\right)e^{i\theta}.
	\end{align*}
Thus,
	\begin{align}
	&(\laplacian A+A)\bar{A_x} \nonumber \\
	&=\left(\laplacian r+r-r\mu^2|\grad r|^2+2i\mu |\grad r|^2+ir\mu \laplacian r+ir\grad r \cdot \grad \mu\right)\left(r_x-ir\theta_x\right) \nonumber \\
	&=\left(\laplacian r+(r+r^3\mu^2-r\mu^2)|\grad r|^2+2i\mu |\grad r|^2+ir\mu \laplacian r+ir\grad r \cdot \grad \mu\right)\left(1-ir\mu\right)r_x, \nonumber
	\end{align}
where we used $1=\left(1+r^2\mu^2\right)|\grad r|^2$ and $\theta_x=\mu r_x$ in the last equality. 

\noindent Taking the imaginary part, we deduce from (\ref{(3)}) in Proposition \ref{PLE} that
	\begin{align}\label{FG1}
	&\mbox{Im}\left((\laplacian A+A)\bar{A_x}\right)=0 \nonumber \\
		&\Leftrightarrow \left(r\grad r \cdot \grad \mu+\left(2\mu-r^2\mu+r^2\mu^3-r^4\mu^3\right)|\grad r|^2\right)r_x=0.
	\end{align}
By a similar computation, we also have
	\begin{align}\label{FG2}
	&\mbox{Im}\left((\laplacian A+A)\bar{A_y}\right)=0 \nonumber \\
		&\Leftrightarrow \left(r\grad r \cdot \grad \mu+\left(2\mu-r^2\mu+r^2\mu^3-r^4\mu^3\right)|\grad r|^2\right)r_y=0.
	\end{align}
Since $|\grad r|\neq 0$, (\ref{FG1}) and (\ref{FG2}) imply that
	\begin{align}\label{FG3}
	r\grad r \cdot \grad \mu+\left(2\mu-r^2\mu+r^2\mu^3-r^4\mu^3\right)|\grad r|^2=0.
	\end{align}
On the other hand, $\theta_{xy}=\mu_yr_x+\mu r_{xy}=\mu_xr_y+\mu r_{yx}=\theta_{yx}$ implies
	\begin{align}
	\mu_yr_x-\mu_xr_y=0. \nonumber 
	\end{align}
Then $\grad \mu$ and $\grad r$ are linearly dependent so that (\ref{FG3}) is equivalent to
	\begin{align}\label{zero}
	r\grad \mu+\left(2\mu-r^2\mu+r^2\mu^3-r^4\mu^3\right)\grad r=0.
	\end{align}
Now we use (\ref{zero}) to compute
	\begin{align}
	&\grad \left(\frac{r^2\mu}{e^{\frac{1}{2}r^2}\sqrt{1+r^2\mu^2}}\right)\nonumber \\
	&=\frac{e^{\frac{1}{2}r^2}\sqrt{1+r^2\mu^2}\left(2r\mu\grad r+r^2\grad \mu\right)-r^2\mu\left(re^{\frac{1}{2}r^2}\sqrt{1+r^2\mu^2}\grad r+e^{\frac{1}{2}r^2}\frac{r^2\mu \grad \mu+r\mu^2 \grad r}{\sqrt{1+r^2\mu^2}}\right)}{e^{r^2}\left(1+r^2\mu^2\right)} \nonumber \\
	&=\frac{\left(1+r^2\mu^2\right)\left(2r\mu\grad r+r^2\grad \mu\right)-r^3\mu\left(1+r^2\mu^2\right)\grad r-r^4\mu^2 \grad \mu-r^3\mu^3 \grad r}{e^{\frac{1}{2}r^2}\left(1+r^2\mu^2\right)^{\frac{3}{2}}} \nonumber \\
	&=\frac{r}{e^{\frac{1}{2}r^2}\left(1+r^2\mu^2\right)^{\frac{3}{2}}}\left(r\grad \mu+\left(2\mu-r^2\mu+r^2\mu^3-r^4\mu^3\right)\grad r\right) \nonumber \\
	&=0. \nonumber
	\end{align}
Therefore we conclude that
	\begin{align*}
	\left(\frac{r^2\mu}{e^{\frac{1}{2}r^2}\sqrt{1+r^2\mu^2}}\right)^2=\frac{r^4\mu^2}{e^{r^2}\left(1+r^2\mu^2\right)}=\frac{r^4|\grad \theta|^2}{e^{r^2}}
	\end{align*}
is a constant. 

For the case $|\grad \theta|\neq 0$, there exists a function $\eta$ such that $\grad r=\eta \grad \theta$. Then, a similar argument yields
	\begin{align}
	\grad \left(\frac{r^2}{e^{\frac{1}{2}r^2}\sqrt{r^2+\eta^2}}\right)=0, \nonumber
	\end{align}
which again implies that
	\begin{align*}
	\left(\frac{r^2}{e^{\frac{1}{2}r^2}\sqrt{r^2+\eta^2}}\right)^2=\frac{r^4}{e^{r^2}(r^2+\eta^2)}=\frac{r^4|\grad \theta|^2}{e^{r^2}}
	\end{align*}
is a constant.
\end{proof}

Next, we prove that the constant in the previous lemma cannot be zero.

\begin{lem}\label{CP}
For $p \in \Sigma^2$ with $|A|>0$, let $r^4|\grad \theta|^2=Ce^{r^2}$ at $p$. Then $C>0$.
\end{lem}

\begin{proof}
Suppose that there exists a point $p \in \Sigma^2$ such that $|A|>0$ and $r^4|\grad \theta|^2=0$ at $p$. Let $\Omega$ be the connected component of $\Sigma^2 \cap \{|A|>0\}$ which contains $p$.
Then Lemma \ref{LC} implies that $\Omega$ is an open set and locally every point shares the same constant in the above lemma. From the connectedness of $\Omega$, we conclude that all points of $\Omega$ share the same constant. Hence $r^4|\grad \theta|^2=0$ in $\Omega$.

Since $r=|A|>0$ in $\Omega$, we have $|\grad \theta|=0$ in $\Omega$. Let $\gamma=\gamma(t)$ be the integral curve of $\grad r$ with $\gamma(0)=p$. Then
	\begin{align}
	r(\gamma(T))-r(\gamma(0))&=\int_0^T\frac{\d}{\d t}r(\gamma(t))\d t \nonumber \\
		&=\int_0^T\grad r(\gamma(t))\cdot \dot{\gamma}(t)\d t \nonumber \\
		&=\int_0^T|\grad r|^2\d t \ =\ T,\nonumber
	\end{align} 
where we used $|\grad r|^2=|\grad r|^2+r^2|\grad \theta|^2=1$ in the last equality.

If the integral curve stays inside $\Omega$, then $r$ increases along the curve by the amount of $T$ increases. This is impossible since $r$ is bounded on $\Sigma^2$. Thus, we may deduce that $\gamma(t)$ approaches the boundary point of $\Omega$. However, by the definition of $\Omega$, $r$ should vanish at the boundary point. This is also a contradiction as $r$ is continuous and increases along $\gamma$. Therefore, such $p\in \Sigma^2$ does not exist.
\end{proof}

Finally, we obtain the following proposition.

\begin{prop}[Global Relations]\label{GC}
Let $F : \Sigma^2 \to \C^2$ be a closed Lagrangian self-shrinker in $\mathscr{F}$, given by $F=(A, B)$. Then, $|A|>0$ and $|B|>0$ on $\Sigma^2$. Moreover, there exist constants $C_1,\ C_2>0$ such that
	\begin{align}
	r_1^4|\grad \theta_1|^2=C_1e^{r_1^2},\ r_2^4|\grad \theta_2|^2=C_2e^{r_2^2}
	\end{align}
on all of $\Sigma^2$, where $A=r_1e^{i\theta_1}$ and $B=r_2e^{i\theta_2}$ are polar representations.
\end{prop}

\begin{proof}
If $|A|$ vanishes at every point of $\Sigma^2$, then $\Sigma^2$ should be contained in a $2$-plane. It is impossible, so there exists a point $p \in \Sigma^2$ with $|A|>0$. Then, by Lemma \ref{LC} and \ref{CP}, there exists a positive constant $C_1$ such that $r_1^4|\grad \theta_1|^2=C_1e^{r_1^2}$ in a neighborhood of $p$. 

Let $U$ be the maximal subset of $\Sigma^2$ consisting of points which share the same constant $C_1$, i.e., $r_1^4|\grad \theta_1|^2=C_1e^{r_1^2}$ in $U$. Since $p\in U$ by the definition of $C_1$, $U$ is non-empty. Moreover, by Lemma \ref{LC}, there exists a neighborhood for each point $q\in U$ such that all points in the neighborhood share the same constant. Thus, $U$ is an open set.

Next, we prove that $U$ is also a closed subset by proving $\partial U=\emptyset$. Suppose the contrary. That is, assume that $x\in \partial U$. If $r_1(x)>0$, then by Lemma \ref{LC} and \ref{CP}, there exists a neighborhood of $x$, $V$, and a positive constant $C_0$ such that $r_1^4|\grad \theta_1|^2=C_0e^{r_1^2}$ in $V$. By the definition of the boundary point, we have $V\cap U \neq \emptyset$ and $V\cap U^c\neq \emptyset$. From $V\cap U \neq \emptyset$, we deduce that $C_0=C_1$. This implies that $V\subseteq U$ and $V\cap U^c=\emptyset$, which is a contradiction. Thus, $r_1(x)=0$. 

On the other hand, from $|\grad r_1|^2+r_1^2|\grad \theta_1|^2=1$ we derive
	\begin{align*}
	|\grad \frac{1}{2}r_1^2|^2=r_1^2-C_1e^{r_1^2}\geq 0
	\end{align*}
in $U$. The last inequality holds, if and only if
	\begin{align*}
	C_1\leq \frac{1}{e}
	\end{align*}
and
	\begin{align*}
	r_{min}(C_1)\leq r_1 \leq r_{max}(C_1),
	\end{align*}
where $r_{min}(C_1)$ and $r_{max}(C_1)$ are two solutions of $r^2=C_1e^{r^2}$.

By the continuity of $r_1$, we conclude that $0=r_1(x)\geq r_{min}(C_1)>0$, which is a contradiction. Therefore $\partial U=\emptyset$ and $U$ is closed. 

Since $U$ is a non-empty open and closed subset in the connected surface, $U$ must be equal to $\Sigma^2$ and we deduce that $r_1^4|\grad \theta_1|^2=C_1e^{r_1^2}$ on all of $\Sigma^2$. The same method can be applied to $B$ and we finish the proof.
\end{proof}

\begin{rmk}
\emph{Although $\theta_1$ and $\theta_2$ are defined up to multiples of $2\pi$, $\grad \theta_1$ and $\grad \theta_2$ are well-defined on $\Sigma^2$.}
\end{rmk}

Next, we compute explicit constants on products of Abresch-Langer curves. All self-shrinkers in this family have reflection symmetry so that Proposition \ref{GC} can be applied. We observe that each constant coincides with the one that appeared in Lemma \ref{TR}. The precise computation would be done as follows.

Let $\Gamma_1$ and $\Gamma_2$ be the Abresch-Langer curves associated with the constants $c_{\Gamma_1}$ and $c_{\Gamma_2}$ as in Lemma \ref{TR}. Suppose $\Gamma_1$ and $\Gamma_2$ are given by $\gamma_1(s)$ and $\gamma_2(t)$, respectively, where $s$ and $t$ are arc-length parametrizations. Consider the Lagrangian self-shrinker $\Gamma_1 \times \Gamma_2 \subset \R^2\times\R^2=\R^4$. If we write $\gamma_1(s)=r_1(s)e^{i\theta_1(s)}$, then
	\begin{align}\label{ap}
	|\dot{\gamma_1}|^2=\dot{r_1}^2+r_1^2\dot{\theta_1}^2=1,
	\end{align}
where the upper dot denotes the derivative with respect to $s$. As $\Gamma_1 \times \Gamma_2\subset \R^4$ is given by the product, we have
	\begin{align*}
	|\grad \theta_1|^2=\dot{\theta_1}^2
	\end{align*}
so that 
	\begin{align}\label{Const}
	C_1\coloneqq \frac{r_1^4|\grad \theta_1|^2}{e^{r_1^2}} =\frac{r_1^4\dot{\theta_1}^2}{e^{r_1^2}}=\frac{r_1^2(1-\dot{r_1}^2)}{e^{r_1^2}},
	\end{align}
where we have used (\ref{ap}) in the last step.
A direct computation gives
	\begin{align*}
	\gamma_1^{\perp}=\gamma_1-\mbox{Re}(\gamma_1\bar{\dot{\gamma_1}})\dot{\gamma_1}=\big{(}r_1(1-\dot{r_1}^2)-ir_1^2\dot{r_1}\dot{\theta_1}\big{)}e^{i\theta_1},
	\end{align*}
and then with (\ref{ap}), we obtain
	\begin{align}\label{curva}
	k_{\Gamma_1}^2=|\ddot{\gamma_1}|^2=|\gamma_1^{\perp}|^2=r_1^2(1-\dot{r_1}^2)^2+r_1^4\dot{r_1}^2\dot{\theta_1}^2=r_1^2(1-\dot{r_1}^2).
	\end{align}
We deduce from (\ref{Const}) and (\ref{curva}) that
	\begin{align*}
	C_1=\frac{r_1^2(1-\dot{r_1}^2)}{e^{r_1^2}}=\frac{k_{\Gamma_1}^2}{e^{r_1^2}}=c_{\Gamma_1}^2,
	\end{align*}
and similarly $C_2=c_{\Gamma_2}^2$.

\section{Proof of the Main Results}
\setcounter{equation}{0}
In this section, we prove the main results of this paper. 

\begin{proof}[Proof of Theorem \ref{MTHM}]
We may assume that $F\in \mathscr{F}$. Let $p \in \Sigma^2$ and suppose that the immersion is given by $F(x, y)=(A(x ,y), B(x, y))$, where $x$ and $y$ are isothermal coordinates near $p$. By Propostion \ref{GC}, $|A|$ and $|B|$ never vanish on $\Sigma^2$ and we may consider polar representations in a neighborhood of $p$ as follows:
	\begin{align*}
	A(x, y)=r_1(x, y)e^{i\theta_1(x, y)},\ B(x, y)=r_2(x, y)e^{i\theta_2(x, y)}.
	\end{align*}
Again by Proposition \ref{GC}, we know that $|\grad \theta_1|>0$ and $|\grad \theta_2|>0$.

Then, as in the proof of Lemma \ref{LC}, there exist functions $\eta_1$ and $\eta_2$ such that
	\begin{align}\label{lind}
	\grad r_1=\eta_1\grad \theta_1,\ \grad r_2=\eta_2 \grad \theta_2.
	\end{align}
From (\ref{(1)}) and (\ref{(2)}) in Proposition \ref{PLE}, we derive
	\begin{align}
	|A_x|=|B_y| &\Leftrightarrow (\eta_1^2+r_1^2)\theta_{1x}^2=(\eta_2^2+r_2^2)\theta_{2y}^2, \nonumber \\
	|A_y|=|B_x| &\Leftrightarrow (\eta_1^2+r_1^2)\theta_{1y}^2=(\eta_2^2+r_2^2)\theta_{2x}^2, \nonumber \\
	A_x\bar{A_y}+B_x\bar{B_y}=0 &\Leftrightarrow (\eta_1^2+r_1^2)\theta_{1x}\theta_{1y}+(\eta_2^2+r_2^2)\theta_{2x}\theta_{2y}=0, \nonumber   
	\end{align}
which gives
	\begin{align*}
	\theta_{1x}\theta_{2x}+\theta_{1y}\theta_{2y}=0.
	\end{align*}
Therefore we conclude that
	\begin{align}\label{ortho}
	\grad \theta_1 \cdot \grad \theta_2=0.
	\end{align}

Since $\grad \theta_1$ and $\grad \theta_2$ are non-vanishing and orthogonal, $\theta_1$ and $\theta_2$ give rise to local coordinates near $p$. In terms of $\theta_1$ and $\theta_2$, we may write $A=r_1(\theta_1, \theta_2)e^{i\theta_1}$. Then, by (\ref{lind}) and (\ref{ortho}),
	\begin{align*}
	\grad r_1 \cdot \grad \theta_2=\eta_1\grad \theta_1 \cdot \grad \theta_2=0.
	\end{align*}
This proves that $r_1$ is independent of $\theta_2$. Similarly
	\begin{align*}
	\grad r_2 \cdot \grad \theta_1=\eta_2 \grad\theta_2 \cdot \grad\theta_1=0,
	\end{align*}
and $r_2$ is independent of $\theta_1$.

Therefore we proved that $\Sigma^2$ is locally given by the product of two curves, parametrized by $\theta_1$ and $\theta_2$, respectively, in a neighborhood of $p$. This implies that $\Sigma^2$ is flat at $p$. Since $p$ was arbitrary, we conclude that $\Sigma^2$ is flat. Then, by the result of Li and Wang (see Proposition \ref{Flat}), $\Sigma^2$ is the product of two Abresch-Langer curves. 
\end{proof}

The unit circle is the only embedded curve in the examples of Abresch and Langer. Therefore we can characterize the Clifford torus as the unique embedded self-shrinker in the products of Abresch-Langer curves, and Theorem \ref{COR} follows directly from Theorem \ref{MTHM}.

\bibliographystyle{abbrv}
\bibliography{CLSSRSREF.bib}

\end{document}